\documentclass[smallextended,envcountsect]{svjour3} 
\usepackage[utf8]{inputenc}
\smartqed 
\usepackage{amsfonts}
\usepackage{amsmath,amssymb}
\usepackage{amsopn}
\usepackage{cite}

\journalname{}

\DeclareMathOperator{\cl}{cl}

\DeclareMathOperator{\dist}{dist}

\DeclareMathOperator{\dom}{dom}
\DeclareMathOperator{\cone}{cone}
\DeclareMathOperator{\conv}{conv}
\DeclareMathOperator{\argmin}{argmin}
\DeclareMathOperator{\tr}{tr}
\newcommand{\R}{\bar{\mathbb{R}}}

\newcommand{\fre}[1][M]{\hat{\partial}_\mathcal{#1}}
\newcommand{\nor}[2][M]{\hat{N}^{\mathcal{#1}}_{#2}}
\newcommand{\coti}[2][M]{\hat{T}^{\mathcal{#1}}_{#2}}
\newcommand{\limsub}[1][M]{\partial_\mathcal{#1}}
\newcommand{\limnor}[2][M]{N^{\mathcal{#1}}_{#2}}

\begin{document}

% All acknowledgements should be placed in the back of the
% paper before References.

\title{Nonconvex Weak Sharp Minima on Riemannian Manifolds}

\author{M.\ M. Karkhaneei \and N. Mahdavi-Amiri}

\institute{Mohammad Mahdi Karkhaneei \at
  Sharif University of Technology  \\
  Tehran, Iran\\
  karkhaneei@gmail.com \and Nezam Mahdavi-Amiri \at
  Sharif University of Technology  \\
  Tehran, Iran\\
  nezamm@sharif.edu \and
 }

\date{Received: date / Accepted: date}

\maketitle

\begin{abstract}
  We are to establish necessary conditions (of the primal
  and dual types) for the set of weak sharp minima of a
  nonconvex optimization problem on a Riemannian manifold.
  Here, we are to provide a generalization of some
  characterizations of weak sharp minima for convex problems
  on Riemannian manifold introduced by Li et al. (SIAM J.
  Optim., 21 (2011), pp. 1523--1560) for nonconvex problems.
  We use the theory of the Fr\'echet and limiting
  subdifferentials on Riemannian manifold to give the
  necessary conditions of the dual type. We also consider a
  theory of contingent directional derivative and a notion
  of contingent cone on Riemannian manifold to give the
  necessary conditions of the primal type. Several
  definitions have been provided for the contingent cone on
  Riemannian manifold. We show that these definitions, with
  some modifications, are equivalent. We establish a lemma
  about the local behavior of a distance function. Using the
  lemma, we express the Fr\'echet subdifferential
  (contingent directional derivative) of a distance function
  on a Riemannian manifold in terms of normal cones
  (contingent cones), to establish the necessary conditions.
   As an application, we show how one can use
  weak sharp minima property to model a Cheeger type
  constant of a graph as an optimization problem on a
  Stiefel manifold. 
\end{abstract}

      

\section{Introduction}
In recent decades, extensive research has been carried out
on optimization on manifolds. These studies scatter in
various contexts such as convex \cite{udriste1994convex},
smooth \cite{AbsilTrustRei2007, absil2007optimization,
   Smith1994}, nonsmooth
\cite{HosseiniTrustRegion2016}, and over special manifolds
such as sphere \cite{AdlerNewtonSpin2002}, Stiefel manifolds
\cite{EdelmanStiefel1999}, etc. Particularly, special
attention have been devoted to nonconvex analysis and
optimization on Riemannian manifolds. Several problems in
machine learning, pattern recognition and computer vision,
can be modeled as nonconvex optimization problems on
Riemannian manifolds \cite{MR3160322, MR2678395, MR2913705,
  MR3604647, MR3069099}. Moreover, various numerical
procedures for solving nonconvex optimization problems on
Riemannian manifolds have been designed, such as line search
and trust region methods along with Newton-like methods in
smooth cases (see
\cite{AbsilTrustRei2007,absil2007optimization,AdlerNewtonSpin2002,MR2811294,MR3123829}),
and subgradient decent, gradient sampling, and proximal
algorithms in nonsmooth cases (see
\cite{Bento2015nonconPoriximal, MR3606427}).
  
Distance functions appear in various optimization methods,
such as proximal point methods, penalty methods, etc.
Generalized differential properties of distance functions
play remarkable roles in variational analysis, optimization,
and their applications. The authors of
\cite{BURKE1992,MR2179254} investigated properties of
generalized derivatives of distance functions in linear
space setting. Properties of distance functions on
a Riemannian manifold are not trivially obtained by
generalization of the corresponding properties in the linear
space setting. Li et al. \cite{LiMordukhovichWangYao} gave a
relation for the subdifferential of a convex distance
function in term of a normal cone to the corresponding set
in a Riemannian manifold setting with the non-positive
sectional curvature. They used a comparison result for
geodesic triangles in a Riemannian manifold with the
non-positive sectional curvature as a global property.
Subdifferentials and normal cones are local notions, and we
do not require any global condition on a Riemannian manifold
for investigating local properties of a distance function.
Here, we establish a lemma about local behavior of a
distance function on a manifold (called local distance
lemma). Using this lemma, we express the Fr\'echet
subdifferential (contingent directional derivative) of a
distance function on a Riemannian manifold in terms of
normal cones (contingent cones).

Then,  we investigate the weak sharp minima notion for
nonconvex optimization problems on Riemannian manifolds. We
provide generalization of some characterizations of weak
sharp minima of  \cite{LiMordukhovichWangYao} to the
nonconvex optimization problems on a Riemannian manifold
modeled in a Hilbert space. To the best of our knowledge,
our work is the first study concerning the notion of weak
sharp minima for nonconvex optimization problems on
Riemannian manifolds.

The concept of sharp minima was introduced by Polyak
\cite{PolyakSharpMinima1979} in the case of
finite-dimensional Euclidean space for the sensitivity
analysis of optimization problems and convergence analysis
of some numerical algorithms. Next, Ferris
\cite{FerrisThesis1988} extended the notion of weak sharp
minima to the situation where the optimization problem has
multiple solutions. The concept was
expanded by many authors for convex and nonconvex
optimization problems over finite and infinite dimensional
linear spaces (see
\cite{BurkeWeakSharpII2005,BurkeFerris1993,Mordu2006FreSubCal,NgWeaksharp2003,Ward1999,Ward1994,MordukhovichWeaksharp2012}).

The primary goal of our work here is to present some primal
and dual necessary conditions for the set of weak sharp
minima of a nonconvex optimization problem on Riemannian
manifold. The key ingredient of deriving the necessary
conditions is the representation of a generalized derivative of
a distance function in terms of a cone. We will use
Fr\'echet and limiting subdifferentials along with Fr\'echet
and limiting normal cones on Riemannian manifold to state
some necessary conditions of dual type for the set of weak
sharp minima of a nonconvex optimization problem. To state
some necessary conditions of primal type for the set of weak
sharp minima of a nonconvex optimization problem, we use a
contingent directional derivative and a contingent cone on
Riemannian manifold. The contingent directional derivative,
which we introduce, is closely related to the Fr\'echet
subdifferential. Several definitions have been provided for
the contingent cone on Riemannian manifold; see
\cite{MR2988725} and \cite{ZhuLedyaev2007}. We will show
that these definitions, with some modifications, are
equivalent.

 The remainder of our work is organized as
follows. In Section 2, some needed preliminaries on linear
spaces and metric spaces are given. Also, we recall some
fundamental notions of variational analysis in linear
space setting. In Section 3, we first introduce some basic
notions of Riemannian manifolds. Then, definitions and
properties of the (limiting) Fr\'echet subdifferential and
(limiting) normal cone on Riemannian manifold are provided.
Also, we introduce contingent directional derivative on
Riemannian manifold and recall some definitions of
contingent cones on Riemannian manifold and show that these
definitions, with some modifications, are equivalent.
In Section 4, we establish a local distance
lemma and use it to attain a formula for the Fr\'echet
subdifferential (the directional derivative) of a distance
function a Riemannian manifold in terms of the normal cone
(the continent cone). In Section 5, we establish some
necessary conditions for weak sharp minima in the nonconvex
case on Riemannian manifold. In Section 6, we give an
application of the results of the previous sections. In
Section 7, we provide our concluding remarks. 

\section{Linear and Metric Spaces}
We provide some definitions and symbols required from
variational analysis and topology; the reader is referred to
\cite{PenotBook2013} for more details.
We denote $\R$ as extended real numbers
$\mathbb{R}\cup\{\infty\}$. Suppose that $E$ is a Banach
space. Then, $E^*$, $\mathbb{B}_E$, and $I_E$ respectively
denote dual space, closed unit ball, and identity function
on $E$. Denote by $\conv A$ and $\cone A$, respectively, the
convex hull of and the cone of a subset $A \subset E$.
Let $(M,d)$ be a metric space and $A \subset M$. Recall that
$\cl A$ stands for the closure of $A$, and the distance
function for subset $A$ is defined by
$\dist(p;A) := \inf_{u\in M}d(p,u)$, for all $p\in M$. The
closed ball with center $p$ and radius $r >$ 0 is denoted by
$\mathbb{B}(p,r) := \{ q \in M \ | \ d(p,q) \leq r \}$.
Moreover, if $f:M \to \R$ is a function on $M$, then we
denote $\dom(f) := \{p \in M : f (p) < \infty\}$ and
$\ker(f) := \{ p \in M : f(p) = 0 \}$. We say that $f$ is
proper, if $\dom(f) \neq \emptyset$. Also, we say that $f$
is lower semicontinuous (l.s.c.) at $p\in M$, if
$f(p) \leq \liminf_{u\to p}f(u)$. We say that $f$ is locally
Lipschitz at $p \in M$ with rate $r$, if there exits a
neighbor of $p$, say $U$, such that
$| f(u) - f(p) | \leq r d(u,p)$, for all $u \in U$.
Furthermore, we say that $f$ is a Lipschitzian function with
rate $r$ (around $q\in \mathcal{M}$), if the earlier
inequality holds for all $p,u \in M$ (for all $p$ and $u$ in
a neighborhood of $q$). Let $\Omega \subset M$. The
indicator function of $\Omega$ is defined as
$\delta_{\Omega}(u):=0$ for $u \in \Omega$, and
$ \delta_{\Omega}(u) := \infty$ for $u \notin \Omega$.
Now, we recall definition of weak sharp minima on a metric
space.
\begin{definition}[weak sharp minima]
  Let $f\colon M \to \R$ be a proper function on a metric
  space $M$ and $S \subset M$. We say that
  \begin{enumerate}
  \item $p \in \Omega$ (where $\Omega := \argmin_S f$) is a
    local weak sharp minimizer for the problem
    $ \min_{ u \in S} f(u)$ with modulus $\alpha>0$, if
    there is $\epsilon>0$ such that for all
    $u \in S \cap \mathbb{B}(p,\epsilon)$, we have
    \begin{equation}
      \label{eq:3}
      f(u) \geq f(p) + \alpha \dist(u; \Omega).
    \end{equation}
  \item $\Omega := \argmin_S f$ is the set of (global) weak
    sharp minima for the problem $ \min_{ u \in S} f(u)$
    with the modulus $\alpha > 0 $, if (\ref{eq:3}) holds
    for all $p \in \Omega$ and $u\in S$.
  \end{enumerate}
\end{definition}

In the sequel, we recall definitions of some basic notions
of variational analysis in linear space setting. Suppose
that $f\colon E \to \R$ is a function on a Banach space $E$
and is finite at $p \in E$. The Fr\'echet subdifferential of
$f$ at $p$ is defined to be
\[ \fre[] f(p) := \left\lbrace x^* \in E^{*} \left|
      \liminf_{u\to p} \frac{f(u)-f(p)-\langle x^* , u-p
        \rangle}{\|u-p\|}\geq 0 \right. \right\rbrace.
\]
The limiting subdifferential of $f$ at $p$ is the set
$\limsub[] f(p)$ of all $x^* \in E^*$ such that there is a
sequence of covectors $x_i^* \in \fre[] f(p_i) $ such that
$f(p_i) \to f(p)$, and $\lim x_i^* = x^*$. The contingent
directional derivative of $f\colon E \to\bar{\mathbb{R}}$ at
$p$ in the direction $v\in E$ is defined as follows:
\begin{equation}
  f^-(p;v):=
  \liminf_{w \to v, t\downarrow 0}\frac{f(p+ t w)-f(p)}{t}.
\end{equation}
Suppose that $\Omega$ is a subset of the Banach space $E$
and $p \in \cl \Omega$. The Fr\'echet normal cone of
$\Omega$ at $p$ is defined to be
\[
  \nor[]{\Omega}(p) := \left\lbrace x^* \in E^{*} \left|
      \limsup_{u \xrightarrow{\Omega} p}\frac{\langle x^* ,
        u-p\rangle}{\|u-p\|} \leq 0 \right. \right\rbrace.
\]
The limiting normal cone of $\Omega$ at $p$ is the set
$\limnor[]{\Omega}(p)$ of limits of sequences
$\{x^*_i\} \subset E^*$ for which there is a sequence
$\{p_i\} \subset \Omega$ converging to $p$ such that
$x^*_i \in \nor[]{\Omega}(p_i)$, for all $i$. The contingent
cone of $\Omega$ at $p$ is defined as
\[
  \coti[]{\Omega}(p) := \{v \in E : \exists v_i \in E, t_i
  \downarrow 0 \text{ such that } v_i \to v, p+t_i v_i) \in
  \Omega \label{def:contiLin} \}.
\]

\begin{proposition}
  \label{pro-nordistLin}
  For a subset $\Omega$ of $E$, with $E$ being a Banach
  space, its distance function $\dist(\cdot;\Omega)$, and
  $p \in \cl\Omega$, one has
  \begin{align}
    \label{eq:DistNormalLin}
    \fre[]\dist(p;\Omega)  =\nor[]{\Omega}(p) \cap \mathbb{B}_{E^{*}}.
  \end{align}
\end{proposition}

The following theorem relates directional derivative of the
distance function with contingent cone.

\begin{proposition}[\cite{BURKE1992}] \label{pro:dirdisLin}
  Suppose that $E$ is a Banach space and
  $p \in \Omega \subset E$. Then, for all $v \in E$, we have
  \begin{equation}\label{eq:dirdisLin}
    d^-_{\Omega}(p;v) \leq
    \dist(v; \coti[]{\Omega}(p)),
  \end{equation}
  where $d^-_{\Omega}(p;v)$ is the contingent directional
  derivative of $\dist(\cdot;\Omega)$ at $p$ in direction
  $v$. If it is further assumed that $ E$ is finite
  dimensional, then equality holds in \eqref{eq:dirdisLin}.
\end{proposition}

\section{Variational Analysis on Riemannian Manifolds}
Here, we recall some definitions and results about
variational analysis on Riemannian manifolds which will be
useful later on; see, e.g.,
\cite{AzagraFerreraLopez2005,ZhuLedyaev2007} for more
details. We will be dealing with functions defined on
Riemannian manifolds (either finite or
infinite-dimensional). A Riemannian manifold
$(\mathcal{M}, g)$ is a $C^\infty$ smooth manifold
$\mathcal{M}$ modeled on some Hilbert space $\mathbb{H}$
(possibly infinite-dimensional), such that for every
$p \in \mathcal{M}$ we are given a scalar product
$g(p) = g_p := \langle \cdot ,\cdot \rangle_p$ on the
tangent space $T_p\mathcal{M} \simeq \mathbb{H} $ so that
$ \|x\|_p =\langle x,x \rangle _p^{1/2}$ defines an
equivalent norm on $T_p\mathcal{M}$ for each
$p \in \mathcal{M}$ and in such a way that the mapping
$p \in \mathcal{M} \to g_p \in \mathcal{S}^2(\mathcal{M}) $
is a $C^\infty $ section of the bundle of symmetric bilinear
forms. For each $p \in \mathcal{M}$, the metric $g$ induces
a natural isomorphism between $T_p\mathcal{M}$ and
$T^{*}_{p}\mathcal{M}$. So, we define the norm on
$T_p^{*}\mathcal{M}$ as $\|v^{*}\|_p^2 = g_p(v,v)$. If a
function $f : \mathcal{M} \to \mathbb{R}$ is (Fr\'echet)
differentiable at $p \in \mathcal{M}$, then norm of the
differential $ df(p) \in T^*_p\mathcal{M}$ at the point $p$
is defined by
$ \| df(p)\|_p := \sup \{df(p)(v) | v \in T_p\mathcal{M},
\|v\|_p \leq 1 \}. $ Given two points
$p, q \in \mathcal{M}$, the Riemannian distance from $p$ to
$q$ is denoted by $d_{\mathcal{M}}(p, q)$. Throughout our work here,
$\mathcal{M}$ is a Riemannian manifold modeled on a
Hilbert space.

Since Fr\'echet/limiting subdifferential and
Fr\'echet/limiting normal cone are local notions in linear
space setting, we can define these notions on Riemannian
manifold, by using a local chart. These definitions are
independent of the chosen chart; see, e.g.,
\cite{AzagraFerreraLopez2005,ZhuLedyaev2007,Pavel1982} for
more details. Here, we define these concepts on Riemannian
manifold by using exponential charts. Suppose that
$f\colon \mathcal{M} \to \R$ is a function on a Riemannian
manifold $\mathcal{M}$ and is finite at $p \in \mathcal{M}$.
The Fr\'echet subdifferential and the limiting
subdifferential of $f$ at $p$, respectively, are defined to
be $\fre f(p) := \fre[] (f \circ \exp_{p})(0) $ and
$\limsub[] f(p) := \limsub[] (f \circ \exp_p)(0)$, where
$\exp_{p}:U \to \mathcal{M}$ is the exponential map of
$\mathcal{M}$ defined on $U$, which is a sufficiently small
neighborhood of $0$ in $T^{*}_p\mathcal{M}$. Similarly, the
contingent directional derivative of
$f\colon \mathcal{M} \to \R$ at $p$ in the direction
$v\in T_p^{*}\mathcal{M}$ is defined to be
$f^-(p;v) := (f \circ \exp_p)^-(0;v)$. Let
$\Omega \subset \mathcal{M}$ and $p \in \cl \Omega$. The
Fr\'echet normal cone and the limiting normal cone of
$\Omega$ at $p$, respectively, are defined to be
$\nor{\Omega}(p) := \nor[]{\exp_{p}^{-1}\Omega_{*}}(0)$ and
$\limnor{\Omega}(p) := \limnor[]{\exp_p^{-1}\Omega_{*}}(0)$,
where $\Omega_{*}$ is the intersection of $\Omega$ with an
arbitrarily small neighborhood of $p$ on which $\exp_p^{-1}$
is defined. Recall that the Fr\'echet normal cone of
$\Omega$ at point $p$ is equal to the Fr\'echet
subdifferential of the indicator function of $\Omega$ at
$p$, that is,
\begin{eqnarray}
  \label{eq:NormalIndicator}
  \nor{\Omega}(p)=\fre  \delta_\Omega(p).
\end{eqnarray}
Similarly, for the limiting subdifferential and normal cone,
we have $ \limnor{\Omega}(p) = \limsub \delta_{\Omega}(p). $
Suppose that $\mathcal{M}$ is a submanifold of a Euclidean
space $E$, $\Omega \subset \mathcal{M}$ and
$p \in \cl \Omega$. Then, from the definition, we have
\begin{equation}\label{eq:17}
\nor{\Omega}(p) = \hat{N}_{\Omega}^{E}(p)  \cap T_p^{*}\mathcal{M}.
\end{equation}
The following proposition
states that the Fr\'echet subdifferential on Riemannian
manifold has the homotone property. The proof is directly
obtained using the definition.

\begin{proposition}[homotone property of subdifferential]
  \label{pro:homotone}
  Consider functions
  $f,g\colon\mathcal{M} \to \mathbb{\bar{R}}$ on a
  Riemannian manifold $\mathcal{M}$. Suppose that $f$ and
  $g$ are finite at $p\in \mathcal{M}$, $g \leq f$ and
  $f(p)=g(p)$. Then, we have $\fre g(p) \subset \fre f(p)$.
\end{proposition}

\subsection{Contingent Cone}
In the sequel, we define contingent cone on a Riemannian
manifold. Similarly, one can define contingent cone on a
Riemannian manifold by means of the exponential function and
the corresponding definition in linear space setting. But,
at first, for a general subset $\Omega$ of a Riemannian
manifold $\mathcal{M}$ and a point $p \in \cl \Omega$, the
notion of contingent cone (the Bouligand tangent) was
introduced by Ledyaev and Zhu \cite[Definition
3.8]{ZhuLedyaev2007}, as all tangent vectors
$v\in T_p\mathcal{M}$ so that
\begin{equation}
  \mbox{there  exist a sequence $t_i \downarrow 0$ and
    $v_i \in T_p\mathcal{M}$ such that  $v_i\to v$ and
    $c_{v_i}(t_i)\in\Omega$,}\label{eq:8}
\end{equation}
where $c_{v_i}$ is an integral curve on $\mathcal{M}$ with
$c_{v_i}(0)=p$ and $c'_{v_i}(0)=v_i$. Li et al. \cite[Remark
3.6]{LiMordukhovichWangYao} mentioned that this definition
is incomplete and required some additional restrictions on
the sequence $\{c_{v_i}\}$, which were in fact used by
Ledyaev and Zhu \cite[Proposition 3.9]{ZhuLedyaev2007}; they
are
\begin{equation}
  \mbox{ $\lim_i c_{v_i}(t_i) = p$ and
    $\lim_i c'_{v_i}(t_i)=v$.}\label{eq:6}
\end{equation}
‌‌But, it seems that these additional conditions are still
insufficient. We should add more conditions, until the
conclusions in \cite[Remark 3.6]{LiMordukhovichWangYao} and
\cite[Proposition 3.9]{ZhuLedyaev2007} hold true. That
condition is,
\begin{equation}
  \mbox{uniform convergence of $c'_{v_i}$ to $c'_{v}$,}\label{eq:7}
\end{equation}
in the sense that for every scalar function
$g \in C^1(\mathcal{M})$, $d(g \circ c_{v_i})$ uniformly
convergences to $d(g \circ c_{v})$ on a neighborhood of $0$.
With these additional conditions, we can show that the
definition of contingent cone as given in \cite[Definition
3.8]{ZhuLedyaev2007} reduces to the following definition
(see Remark \ref{rem:twodef} below).

\begin{definition}[contingent cone] \label{def:conti}
  Suppose that $\Omega$ is a subset of the Riemannian
  manifold $\mathcal{M}$ and $p \in \cl \Omega$. The
  contingent cone of $\Omega$ at $p$ is defined as
  \[
    \coti{\Omega}(p) := \{v \in T_p\mathcal{M} : \exists v_i
    \in T_p\mathcal{M}, t_i \downarrow 0 \text{ such that }
    v_i \to v, \exp_p(t_i v_i) \in \Omega \}.
  \]
\end{definition}

It is worthwhile to mention that Definition \ref{def:conti}
has been used by Hosseini and Pouryayevali \cite{MR2988725}.

\begin{remark}\label{rem:twodef}
  Here, we show that Definition \ref{def:conti} is
  equivalent to \cite[Definition 3.8]{ZhuLedyaev2007} when
  we add additional restrictions (\ref{eq:6}) and
  (\ref{eq:7}) on a sequence $\{c_{v_i}\}$ for a vector
  $v_i \in T_p\mathcal{M}$ therein. Our proof makes use of a
  technique of \cite[Proposition
  3.5]{LiMordukhovichWangYao}. Let $T_{\Omega}(p)$ be the
  set of all tangent vectors $v \in T_p\mathcal{M}$ for
  which the conditions (\ref{eq:8}), (\ref{eq:6}) and
  (\ref{eq:7}) hold true. We want to show that
  $\coti{\Omega}(p) = T_{\Omega}(p)$. Obviously,
  $\coti{\Omega}(p) \subset T_{\Omega}(p)$. We will show the
  reverse of the inclusion. Let $v \in T_{\Omega}(p)$. By
  (\ref{eq:8}), there exist a sequence $t_i \downarrow 0$
  and $v_i \in T_p\mathcal{M}$ such that $v_i\to v$ and
  $c_{v_i}(t_i)\in\Omega$. Denote
  $w_i := \exp_p^{-1}c_{v_i}(t_i)$ for sufficiently small
  $t_i$. Then, we show that
  \begin{equation}\label{eq:contiEquiv}
    v = \lim_{t_i \to 0}\frac{w_i}{t_i} ,
  \end{equation}
  which clearly shows that $v \in \coti{\Omega}(p)$. To
  establish \eqref{eq:contiEquiv}, take
  $f \in C^1(\mathcal{M})$ and, by its smoothness, obtain
  $ f(c_{v_i}(t_i))=f(\exp_p w_i)=f(p)+\langle df(p), w_i
  \rangle + o(\|w_i\|), $ which in turn implies
  \begin{equation}\label{eq:contiEquiv2}
    \langle df(p), v \rangle= \lim_{t_i \to 0+} \frac{f(c_{v_i}(t_i))-f(p)}{t_i} =  
    \lim_{t_i \to 0+} \langle df(p), \frac{w_i}{t_i} \rangle +  \lim_{t_i \to 0+}\frac{o(\|w_i\|)}{t_i}.
  \end{equation}
  Since $c'_{v_i}$ uniformly converges to $c'_v$, for a
  constant $L>0$, we have
  \[
    \|w_i\| = d(c_{v_i}(t_i),p) \leq
    \int_{0}^{t_i}\|c'_{v_i}(t)\|dt \leq L t_i.
  \]
  We get that $\frac{\|w_i\|}{t_i}$ is bounded as
  $t_i \to 0+$. The latter, together with $\lim_i w_i= 0$
  and \eqref{eq:contiEquiv2}, implies that
  \eqref{eq:contiEquiv} holds, because
  $ f\in C^1(\mathcal{M})$ was chosen arbitrarily. Thus, the
  proof is complete.
\end{remark}
 
\section{Generalized Derivatives of a Distance Function}
Generalized differential properties of distance functions
play remarkable roles in variational analysis, optimization,
and their applications. Generalized derivatives of distance
function have fundamental roles in the analysis of
optimization algorithms, such as Proximal point methods in
both linear space setting and Riemannian manifold. The
authors of \cite{BURKE1992,MR2179254} investigated
properties of generalized derivatives of distance functions
in linear space setting. Properties of distance functions on
a Riemannian manifold are not trivially obtained by
generalization of the corresponding properties in linear
spaces setting.

The following statement shows a relation between distance of
two points on a Riemannian manifold and distance of image of
two points under a chart.
\begin{proposition}[\cite{MR2569498}]
  \label{pro:chartdist}
  For
  any point $p \in \mathcal{M}$ and chart $(U', \psi)$
  around $p$ there exist a $U \subseteq U'$ and a constant
  $C \ge 1$ such that for all $p,x \in U$, we have
  \begin{equation}\label{eq:10}
    \frac 1 C \| \psi(p) - \psi(x) \| \le d(p,x) \le C \| \psi(p) - \psi(x) \|.
  \end{equation}
\end{proposition}
Now, Suppose that $\Omega \subset \mathcal{M}$ and
$p \in \cl \Omega$. For every $r> 0$, denote
$\Omega_r:=\Omega \cap \mathbb{B}(p,r)$. In Proposition
\ref{pro:chartdist}, by setting $h = \exp_p ^{-1}$ and
taking supremum over ${x \in \Omega_r}$, one
gets\footnote{See: https://mathoverflow.net/q/301064.}
\begin{equation}
  \label{eq:2}
  \frac 1 C \dist( \exp_p ^{-1} (u); \exp_p ^{-1} (\Omega_r))  \le d(u;\Omega_r) \le C \dist( \exp_p ^{-1} (u) ; \exp_p ^{-1} (\Omega_r) ) \ ,
\end{equation}
which is an local estimate of a distance function on
a Riemannian manifold in terms of normal coordinates. By
these estimates, one can obtain some properties of a
distance function on a Riemannian manifold from corresponding
results in linear space setting. For example, by using
homotone property of the Fr\'echet subdifferential and a
well-known result in linear space setting (Proposition
\ref{pro-nordistLin}), we get
\begin{equation}
  \label{eq:1}
  \frac{1}{C}\nor{\Omega}(p) \cap \mathbb{B}_{T_p^*\mathcal{M}}
  \leq \fre \dist(p;\Omega)
  \leq C\nor{\Omega}(p) \cap \mathbb{B}_{T_p^*\mathcal{M}}.
\end{equation} 
Li et al. \cite{LiMordukhovichWangYao} established
\eqref{eq:1} with $C=1$ for a convex subset of a Riemannian
manifold with the non-positive sectional curvature. In their
proof, there were two fundamental points: first, since the
authors used the convex calculus on manifolds, it was
necessary to assume that the manifolds were Hadamard
manifolds and $\Omega$ was a geodesic convex set so that the
distance function $ \dist(\cdot; \Omega)$ be a convex
function; Second, they used a comparison result for geodesic
triangles in Hadamard manifold, which was a global
property. Since we are to use the nonsmooth calculus on
manifolds, we do not need the distance function to be
convex. On the other hand, since the Fr\'echet
subdifferential and normal cone are local notions, we can
prove \eqref{eq:1} with $C=1$ for an arbitrary subset of an
arbitrary Riemannian manifold with a better local estimate
of a distance function than \eqref{eq:2}.

\begin{lemma}[local distance lemma]
  \label{local-distance-lemma}
  Let $\Omega $ be a subset of $\mathcal{M}$ and
  $p \in \cl \Omega$. For sufficiently small $r> 0$, denote
  $\Omega_r := \Omega \cap \mathbb{B}(p,r)$. For each
  $u \in \mathbb{B}(p,r)$, we have
  \begin{equation*}
    \frac{\dist(\exp_p^{-1}u; \exp_p^{-1} \Omega_r)}{1 + |o(r)|}  \leq
    \dist(u; \Omega_r) \leq 
    (1 + |o(r)|)\dist(\exp_p^{-1}u; \exp_p^{-1} \Omega_r).
  \end{equation*}
\end{lemma}

\begin{proof} Let $r$ be sufficiently small such that
  $\exp_p^{-1}$ is defined on $\mathbb{B}(p,r)$. It is
  enough to show that for every
  $u_{1},u_2 \in \mathbb{B}(p,r)$, we have
  \begin{equation}\label{eq:9-1}
    (1 - |o(r)|)d(v_1, v_2) \leq
    d(u_{1}, u_{2}) \leq 
    (1 + |o(r)|)d(v_1, v_2),
  \end{equation}
  where $v_1 := \exp_p^{-1}u_1$ and $v_2 = \exp_p^{-1} u_2$.
  Suppose that $\gamma(t) = v_1 + t(v_2-v_1)$,
  $0\leq t \leq 1$, is the straight line segment joining
  $v_1$ and $v_2$. Denote
  $c(t) := \exp_p\gamma(t), 0 \leq t \leq 1$. Let
  $ f = \exp^{*}_p(g)$ be the pull back of the metric $g$ of
  $\mathcal{M}$ in $\exp_{p}^{-1} \mathbb{B}(p,r)$. We have
  $\dot{c}(t) = (D\exp_p)_{\gamma(t)}(\dot{\gamma}(t))$ and
  \begin{equation*}
    d(u_1,u_2) \leq \int_0^1 \sqrt{g_{c(t)}(\dot{c}(t), \dot{c}(t))}dt =
    \int_0^1 \sqrt{f_{\gamma(t)}(\dot{\gamma}(t),\dot{\gamma}(t))}dt.
  \end{equation*}
  According to \cite[Theorem~5.5]{MR1335233}, we have the
  following Taylor's expansion of $f_v$ at $0$:
  $f_{v} = g_p+ q_v + h_{v}$,  for  $|v| \to 0$,
  where $q_v(w_1,w_2) := \frac 13 R_p(v,w_1,v,w_2)$ is a
  symmetric bilinear function obtained from the Riemann
  curvature $R_p$ of $\mathcal{M}$ and $h_{v}$ is a bilinear
  function on $T_{v}T_p\mathcal{M}$ whose norm is of order
  $O(|v|^{3})$. So, we have
  \begin{eqnarray*}
    d(u_1,u_2) &\leq& \int_0^1\sqrt{g_p(\dot{\gamma}, \dot{\gamma})+q_{\gamma(t)}(\dot{\gamma},\dot{\gamma}) + h_{\gamma(t)}(\dot{\gamma},\dot{\gamma}) } \,dt \\
               &=& \int_0^1
                   \sqrt{ 1
                   +
                   \frac 13 \left(\frac{R_p(\gamma,\dot{\gamma},\gamma,\dot{\gamma})}
                   {g_p(\dot{\gamma}, \dot{\gamma}) g_p(\gamma,\gamma)} \right)g_p(\gamma,\gamma)
                   +
                   \frac{h_{\gamma(t)}(\dot{\gamma}, \dot{\gamma})}
                   {g_p(\dot{\gamma}, \dot{\gamma})}
                   }
                   \sqrt{g_p(\dot{\gamma}, \dot{\gamma})}\,dt
    \\
               &\leq& \big(1 + \frac{\|R_p\|}{6}r^2 + O(r^3)\big)d(v_1,v_2),
  \end{eqnarray*}
  where $\|R_p\|$ is the supremum of $R_p(w_1,w_2,w_3,w_4)$
  over all $w_1,w_2,w_3,w_4 \in T_p\mathcal{M}$ with norms
  equal to $1$. Note that since $R_p$ is a continuous
  multilinear function, $\|R_p\|$ is finite. So, the second
  inequality of \eqref{eq:9-1} is established. To establish
  the first inequality of \eqref{eq:9-1}, we consider an
  arbitrary curve $c$ joining $u_1$ and $u_2$ in
  $\mathbb{B}(p,r)$. Then, using a similar technique as
  above, one can show that the length of $c$ is at least
  $\big(1 - \frac{\|R_p\|}{6}r^2 + O(r^3)\big)d(v_1,v_2)$.
  So, by taking the infimum, the desired inequality is at hand.
  Therefore, the proof is complete. \qed
\end{proof}

\begin{remark}
  The proof of Lemma \ref{local-distance-lemma} gives a
  tighter estimate of $o(r)$. Indeed, we have
  $ \dist(u; \Omega_r)= \big(1 \pm \frac{\|R_p\|}{6}r^2 +
  O(r^3)\big)\dist(\exp_p^{-1}u; \exp_p^{-1} \Omega_r).$
\end{remark}

Now, we can prove \eqref{eq:1} with $C=1$ in a general
setting, which plays an essential role for stating some
necessary conditions of dual type for the set of weak sharp
minima of a nonconvex optimization problem in Section 5.

Note that results of this type, relating subdifferential of
the distance function and normals to the corresponding set,
are known for various subdifferentials in general nonconvex
settings of Banach spaces and are of great importance for
many aspects of variational analysis; e.g., see
\cite{MordukhovichBookI2006,PenotBook2013}.

\begin{theorem}\label{thm-nordist}
  With $\Omega$ a subset of $\mathcal{M}$, for the distance
  function $\dist(\cdot;\Omega)$, and $p \in \cl\Omega$, one
  has
  \begin{align}
    \label{eq:DistNormal}
    \fre\dist(p;\Omega)  =\nor{\Omega}(p) \cap \mathbb{B}_{T_p^*\mathcal{M}}.
  \end{align}
\end{theorem}

\begin{proof}
  By local distance lemma
  (Lemma~\ref{local-distance-lemma}), for sufficiently small
  values of $r$, we have
  $
    \dist(u; \Omega) = \dist(u; \Omega_r) \leq (1 +
    |o(r)|)\dist(\exp_p^{-1}u; \exp_p^{-1} \Omega_r),
  $
   for each $ u\in \mathbb{B}(0,r/2)$.  Now, the homotone
  property of subdifferential and the similar property in
  linear space setting (Proposition \ref{pro-nordistLin})
  imply that
  \begin{eqnarray*}
    \fre \dist(p; \Omega) &\leq& (1 + |o(r)|) \fre \dist(0;
                                 \exp_p^{-1} \Omega_r) \\ & =& 
                                                               (1 + |o(r)|) \nor[]{\exp_p^{-1}\Omega_r}(0)\cap
                                                               \mathbb{B}_{T_p^{*}\mathcal{M}} \\ & =&
                                                                                                       (1 + |o(r)|)\nor[M]{\Omega}(p)\cap
                                                                                                       \mathbb{B}_{T_p^{*}\mathcal{M}}.
  \end{eqnarray*}
  By $r \to 0$, we have
  $\fre \dist(p;\Omega) \leq \nor{\Omega} (p) \cap
  \mathbb{B} _{T^{*}_p\mathcal{M}}$. The reverse of the
  inequality can be established similarly, to complete the
  proof. \qed
\end{proof}
% \begin{remark}
%   When $\mathcal{M}$ is a Hadamard manifold and $\Omega$
%   is a closed and convex subset of $\mathcal{M}$, a
%   similar result of \cref{pro-nordist} was proved by Li et
%   al. \cite[Theorem 5.3]{LiMordukhovichWangYao}. The
%   authors used a nonexpansion property of $\exp_p^{-1}$,
%   which is a global property for Hadamard manifolds. But,
%   normal cone and subdifferential possess local
%   properties. So, we could establish \cref{pro-nordist}
%   for the general Riemannian manifold with the help of
%   local distance lemma (Lemma \ref{local-distance-lemma}).
% \end{remark}

Immediately, we have the following corollary from Theorem
\ref{thm-nordist}.

\begin{corollary}
  With the notation of Theorem~\ref{thm-nordist}, we have
  \begin{align*}
    \nor{\Omega}(p) = \cone   \fre\dist(p;\Omega).
  \end{align*}
\end{corollary}

The following theorem relates directional derivative of the
distance function with contingent cone. It immediately
follows from a corresponding property in linear space
setting by using the local distance lemma. We use this
property to obtain a necessary condition for weak sharp
minima. This proposition follows immediately from a similar
result in linear spaces \cite[Theorem 4]{BURKE1992} by using
the mentioned local distance lemma above.

\begin{theorem}\label{thm:dirdis}
  Suppose that $\mathcal{M}$ is a Riemannian manifold and
  $p \in \Omega \subset \mathcal{M}$. Then, for all
  $v \in T_p\mathcal{M}$, we have
  \begin{equation}\label{eq:dirdis}
    d^-_{\Omega}(p;v) \leq
    \dist(v; \coti{\Omega}(p)),
  \end{equation}
  where $d^-_{\Omega}(p;v)$ is the directional derivative of
  $\dist(\cdot;\Omega)$ at $p$ in direction $v$. If it is
  further assumed that $\mathcal{M}$ is finite dimensional,
  then equality holds in \eqref{eq:dirdis}.
\end{theorem}

\begin{proof} We have
  \begin{equation*} \frac{\dist(\exp_ptw ; \Omega) -
      \dist(p;\Omega)}{t} \leq (1 + |o(r)|) \frac{\dist (t
      w;\exp_p^{-1} \Omega_{r})-0}{t}.
  \end{equation*} By taking $\lim\inf$ as $t \downarrow 0$
  and $w \to v$, and by using a similar property in linear
  space setting (Proposition \ref{pro:dirdisLin}), we have $
  d^-_{\Omega}(p;v)\leq \dist (v;
  \coti[]{\exp_p^{-1}\Omega_r}(0)) = \dist(v;
  \coti{\Omega}(p))$. Similarly, one can prove that if it is
  further assumed that $\mathcal{M}$ is finite dimensional,
  then equality holds in \eqref{eq:dirdis}.
  \qed
\end{proof}

The following corollary is now at hand.
\begin{corollary}
  With the notation of Theorem \ref{thm:dirdis}, if
  $p \in \Omega \subset \mathcal{M}$, then
  \[
    \coti{\Omega}(p) \subset \{ v : d^-_{\Omega}(p;v) \leq
    0\},
  \]
  and equality holds when $\mathcal{M}$ is finite
  dimensional.
\end{corollary}

\begin{remark}
  To prove Theorem \ref{thm-nordist} and Theorem
  \ref{thm:dirdis}, we do not need the underlying Riemannian
  manifold to be a Hadamard manifold. But, for the
  development of some optimization techniques on Riemannian
  manifolds, such as variational inequalities
  \cite{MR1942573} and equilibrium problems
  \cite{MR3192427}, properties of Hadamard manifolds are
  indeed essential; see Krist\'aly \cite[Remark
  5.1]{MR3192427}.
\end{remark}

\section{Nonconvex Weak Sharp Minima on a Riemannian
  Manifold}\label{sec:necess-cond-weak}
Here, we give some necessary conditions of the primal type
and the dual type for the set of weak sharp minima of an
optimization problem on a (possibly infinite dimensional)
Riemannian manifold. Formerly, Li at al.
\cite{LiMordukhovichWangYao} provided some characterizations
of weak sharp minima in the case of convex problems on
finite dimensional Riemannian manifold.

Note that $p\in \Omega$ (where $\Omega := \argmin_S f$)
being a local weak sharp minimizer for the problem
$ \min_{ u \in S} f(u)$ with modulus $\alpha>0$, is
equivalent to $p\in \Omega$ being a local minimizer of the
following perturbed problem:
\begin{equation}
  \label{eq:4}
  \min_{u\in S}(f(u)-\alpha \dist(u; \Omega)).
\end{equation}
Similarly, $\Omega$ being the set of weak sharp minima for
the problem $ \min_{ u \in S} f(u)$ with the modulus
$\alpha > 0 $, is equivalent to $p \in \Omega$ being a
global minimizer of the perturbed problem (\ref{eq:4}). So,
the set of weak sharp minima of an optimization problem is
equivalent to the set of minimizers of a difference
optimization problem.

\begin{remark}
  The properties of weak sharp minima on Riemannian
  manifold could not trivially be advocated as the
  properties of weak sharp minima on a linear space setting.
  Indeed, if we substitute $u$ by $\exp_pw$ in \eqref{eq:4},
  the objective function is converted to
  $(f \circ \exp_p)(w) - \alpha \dist(\exp_pw;\Omega)$. But,
  the local minimum of the converted problem is not
  trivially related to the weak sharp minima of an
  optimization problem on a linear space.
\end{remark}

The homotone property of Fr\'echet subdifferential
(Proposition \ref{pro:homotone}) admits a necessary
optimality condition for a local minimum of a function on a
Riemannian manifold.
\begin{proposition}\label{cor:optimal}
  Let $f\colon\mathcal{M} \to \mathbb{\bar{R}}$ be a
  function on a Riemannian manifold $\mathcal{M}$. Suppose
  that the value of $f$ is finite at $p \in \mathcal{M}$. If
  $p$ is a local minimizer of $f$, then $0 \in \fre f(p)$.
\end{proposition}

Next, we state a simple rule about the Fr\'echet
subdifferential of sum of two functions, which is directly
deduced from the definition.
\begin{proposition}\label{pro-sum}
  Consider functions
  $f_1,f_2\colon\mathcal{M} \to \bar{\mathbb{R}}$ on a
  Riemannian manifold $\mathcal{M}$. Suppose
  $p \in \mathcal{M}$ and $f_1(p),f_2(p) < \infty$. Then, we have
  \begin{equation}
    \label{eq:cal1}
    \fre f_1(p)+ \fre f_2(p) \subset \fre (f_1 + f_2)(p).
  \end{equation}
  Moreover, if one of the $f_i$ is Fr\'echet differentiable,
  then we have equality in (\ref{eq:cal1}).
\end{proposition}

Now, we state some necessary conditions of the primal type
and the dual type for a local weak sharp minimizer of an
unconstrained optimization problem.
\begin{theorem}[necessary conditions for a local weak sharp
  minimizer of an unconstrained problem on a Riemannian
  manifold]
  \label{thm:NC-UnConWSP-local}
  Let $\Omega$ be the solution set for problem (\ref{eq:1}).
  Suppose that $S := \mathcal{M}$ and
  $p \in \Omega := \argmin_S f$ is a local weak sharp
  minimizer for the problem $ \min_{ u \in S} f(u)$ with
  modulus $\alpha>0$. Then, the followings holds.

  \begin{enumerate}
  \item[(i)] We have
    \begin{equation}
      \label{eq:UnConstraintWSM}
      \alpha \mathbb{B}_{T^*_p\mathcal{M}} \cap
      \nor{\Omega}(p) \subset \fre f(p).
    \end{equation}
  \item[(ii)] For all $v \in T_p\mathcal{M}$, we have
    \begin{equation}
      \label{eq:UnConstraintWSM2}
      f^-(p;v) \geq \alpha \dist(v; \coti{\Omega}(p)).
    \end{equation}
  \end{enumerate}
\end{theorem}
\begin{proof}
  By definition, there exists $\epsilon > 0$ such that
  \[
    f(u) \geq f(p) + \alpha \dist(u; \Omega),\qquad \forall
    u\in \mathbb{B}(p,\epsilon).
  \]
  Since the Fr\'echet subdifferential has homotone property
  (Proposition \ref{pro:homotone}), we have
  $ \fre\alpha\dist(p;\Omega) \subset \fre f(p). $ Theorem
  \ref{thm-nordist} implies
  $\alpha \nor{\Omega}(p) \cap \mathbb{B}_{T_p^*\mathcal{M}}
  \subset \fre f(p)$. So, (i) is proved. Next, we prove
  (ii). Let $v \in T_p\mathcal{M}$. The hypothesis
  guarantees that for all $w$ sufficiently close to $v$ and
  for all sufficiently small $t>0$, we have
  $
    f(\exp_p t w) -f(p) \geq \alpha \dist(\exp_p t w ;
    \Omega),
  $
  which implies
  \[
    \frac{f(\exp_p t w) -f(p)}{t} \geq \alpha
    \frac{\dist(\exp_p t w ; \Omega)-\dist(p;\Omega)}{t}.
  \]
  By taking $\liminf$ of both sides of the latter
  inequality, as $w \to v$ and $t \downarrow 0$, and
  applying Theorem \ref{thm:dirdis}, (ii) is obtained. \qed
\end{proof}
% \begin{remark}
%   By getting norm from two sides of the inclusion
%   \eqref{eq:UnConstraintWSM}, immediately, an upper estimate
%   obtain for the modulus $\alpha$. That is
%   $\alpha \leq \sup_{x^{*}\in \fre f(p)} \|x^{*}\|_{p}$.
% \end{remark}
Since every element of the set of global weak sharp minima
is a local weak sharp minimizer, we immediately have the
following corollary.
\begin{corollary}[necessary conditions for unconstrained
  weak sharp minima on a Riemannian manifold]
  \label{pro:NC-UnConWSP}
  Suppose that $S := \mathcal{M}$ and $\Omega := \argmin_Sf$
  is the set of weak sharp minima for the problem $\min_Sf$
  with modulus $\alpha>0$. Then,
  \begin{enumerate}
  \item for every $p \in \Omega$, we have the inclusion
    (\ref{eq:UnConstraintWSM}).
  \item for every $p \in \Omega$ and $v \in T_p\mathcal{M}$,
    we have the inequality (\ref{eq:UnConstraintWSM2}).
  \end{enumerate}
\end{corollary}
 
\begin{remark}
  Similar to Ward \cite{Ward1994}, with some modifications
  of the definition of contingent directional derivative on
  Riemannian manifold, one can state some necessary
  conditions for weak sharp minima of higher orders.
\end{remark}
\begin{remark}
  Similar to the approach of Studniarski and Ward
  \cite{Ward1999} in linear spaces, one can state some
  sufficient conditions for weak sharp minima on Riemannian
  manifold based on a generalization of the contingent
  directional derivative.
\end{remark}

\subsection{Constrained Weak Sharp Minima on Riemannian Manifolds}
In the sequel, we give some necessary conditions for the
weak sharp minima of a constrained optimization problem on a
Riemannian manifold. As said at the beginning of this
section, the set of weak sharp minima is equivalent to the
set of minimizers of a difference optimization problem
(\ref{eq:4}). So, at first, we state some necessary
conditions for minimizers of a difference optimization
problem, i.e., an optimization problem whose cost function
is given in a difference form. To present these necessary
conditions, we use the Fr\'echet and limiting
subdifferentials on a Riemannian manifold. Consider the
difference optimization problem with the geometric
constraint:
\begin{equation}\label{eq:diffGeoCon}
  \min_{u\in S} f(u),
\end{equation}
where $f\colon\mathcal{M}\to\bar{\mathbb{R}}$ may be
represented as $f=f_1-f_2$. With the help of indicator
function on the set $S$, constrained optimization problem
(\ref{eq:diffGeoCon}) can be rewritten in an unconstrained
form:
\begin{equation}\label{eq:diffGeoUnCon}
  \min_{u\in \mathcal{M}} f(u) + \delta_S(u).
\end{equation}
To present some necessary optimality conditions for the
unconstrained optimization problem, we need to decompose the
subdifferential of the objective function of problem
(\ref{eq:diffGeoUnCon}). We say that $f$ is Fr\'echet
decomposable at $p\in S$ on $S$, when
\begin{equation}\label{eq:5}
  (f + \delta_{S})(p)\subset \fre f(p) + \nor{S}(p).
\end{equation}
Note that this definition is a Riemannian manifold
counterpart of the one used by Mordukhovich et al.
\cite{Mordu2006FreSubCal} in the linear space setting. If
$f$ is Fr\'echet differentiable at $p$, then Proposition
\ref{pro-sum} implies that $f$ is Fr\'echet decomposable at
$p\in S$ on $S$. Moreover, by \cite[Proposition
4.3]{LiMordukhovichWangYao}, for a convex function $f$ with
$\dom f$ having nonempty interior and a nonempty convex set
$S$ such that $S \cap \dom f$ is convex, we have that $f$ is
Fr\'echet decomposable on $S$ at every point in
$\mathrm{int}(\dom f)\cap S$. According to the attractive
form of the calculus for the limiting subdifferential of a
Lipschitzian function, decompose condition (\ref{eq:5}) for
the limiting subdifferential is at hand. In the first
assertion of the following theorem, we impose the
decomposition assumption for the Fr\'echet subdifferential
on a Riemannian manifold, while the second part is justified
without this assumption via the limiting subdifferential on
a Riemannian manifold  and its attractive forms in the
Lipschitzian case, i.e., if
$f_1\colon\mathcal{M} \to \bar{\mathbb{R}}$ is Lipschitzian
around $p \in \mathcal{M}$ and
$f_2\colon\mathcal{M} \to \bar{\mathbb{R}}$ is an l.s.c.
function and is finite at $p$, then
\begin{equation}\label{eq:limsubsum}
  \limsub(f_1+f_2)(p) \subset \limsub f_1(p) + \limsub f_2(p).
\end{equation}

\begin{theorem}[necessary conditions for difference problems
  with geometric constrains]
  \label{thm:NC-diffGeoCon}
  Suppose that $p$ is a local solution of
  (\ref{eq:diffGeoCon}), and $f$ is represented as
  $f=f_1-f_2$, where
  $f_i\colon\mathcal{M}\to\bar{\mathbb{R}}$ is finite at
  $p$. Then,
  \begin{enumerate}
  \item[(i)] If $f_1$ is Fr\'echet decomposable at $p\in S$
    on $S$, then we have
    \begin{equation}
      \label{eq:NOCDiffProGeoCon}
      \fre f_2(p)\subset \fre f_1(p)+\nor{S}(p).
    \end{equation}
    Particularly, when $f \equiv -f_2$, we have
    $ \fre f_2(p) \subset \nor{S}(p). $
  \item[(ii)] If $f_1$ is Lipschitzian around $p$ and $S$ is
    closed, then we have
    \[
      \fre f_2(p) \subset \limsub f_1(p) + \limnor{S}(p).
    \]
  \end{enumerate}
\end{theorem}

\begin{proof}
  Problem \eqref{eq:diffGeoCon} can be written as the
  following unconstrained difference problem:
  $
    \min_{x\in \mathcal{M}} [f_1(x)+ \delta_{S}(x) - f_2(x)].
  $
  ‌‌‌By Corollary~\ref{cor:optimal}, we have
  $ 0 \in \fre(f_1(x)+ \delta_{S}(x) - f_2(x))(p). $
  Proposition \ref{pro-sum} implies
  \begin{equation}\label{eq:proofNC-diffGeoCon1}
    \fre f_2(p)\subset \fre (f_1 + \delta_{S})(p).
  \end{equation}
  If $f_2$ is Fr\'echet decomposable at $p \in S$ on $S$,
  then we have $\fre f_2(p)\subset \fre f_1(p)+\nor{S}(p)$.
  So, (i) is proved. Now, we prove part (ii). Since $S$ is
  assume to be closed, $\delta_{S}$ is l.s.c and the
  inclusion \eqref{eq:proofNC-diffGeoCon1} with the sum rule
  for the limiting subdifferential (\ref{eq:limsubsum}) imply
  \begin{align*}
    \fre f_2(p) &\subset  \fre (f_1 + \delta_{S})(p)
                  \subset \limsub (f_1 + \delta_{S})(p) \\
		&\subset \limsub f_1(p) + \limsub \delta_{S}(p) =
           \limsub f_1(p) + \limnor{S}(p),
  \end{align*}
  and the proof of (ii) is complete. \qed
\end{proof}

In the following, as a corollary of
Theorem~\ref{thm:NC-diffGeoCon}, we can state some necessary
conditions for the set of weak sharp minima of a constrained
optimization problem. The result can be established using
Theorem~\ref{thm:NC-diffGeoCon} by reducing weak sharp
minima to minimizers of a constrained difference
optimization problem.
\begin{corollary}[necessary conditions for weak sharp minima
  under geometric constraints on Riemannian manifolds]
  Suppose $\Omega := \argmin_S f$ is the set of weak sharp
  minima for the problem $\min_Sf$ with modulus $\alpha>0$.
  Then, we have
  \begin{enumerate}
  \item If $p\in \Omega$ and $f$ is Fr\'echet decomposable
    on $S$ at the point $p$, then
    \begin{equation*}
      \alpha \mathbb{B}_{T^*_p\mathcal{M}} \cap
      \nor{\Omega}(p) \subset \fre f(p)+\nor{S}(p).
    \end{equation*}
  \item If $f$ is Lipschitzian around $p\in S$ and $S$ is
    closed, then we have
    \begin{equation*}
      \alpha \mathbb{B}_{T^*_p\mathcal{M}} \cap
      \nor{\Omega}(p) \subset \limsub
      f(p)+\limnor{S}(p).
    \end{equation*}
  \end{enumerate}
\end{corollary}

\section{Application}
 
Here, we give an application of the nonconvex weak sharp
minima on Riemannian manifolds. We will show how the notion
of weak sharp minima and the results of the previous
sections can be used to model a discrete optimization
problem as an unconstrained optimization problem on a
Stiefel manifold. Recall that for integers $0<k \leq n$, the
Stiefel manifold $St(n,k)$ is defined as the set of all
matrices $U \in \mathbb{R}^{n\times k}$, with $U^TU=I_k$,
where $U^T$ denotes the transpose of $U$ and $I_k$ denotes
the $k \times k$ identity matrix. For each $P \in St(n,k)$,
we have
\begin{equation*}
  T_{P}St(n,k) =
  \{ X \in \mathbb{R}^{n \times k} \, |\,
  X^{T}P + P^{T}X = 0 \},
\end{equation*}
and we endow $T_{P}St(n,k)$ with the induced metric of the
Euclidean space $\mathbb{R}^{n\times k}$:
$\langle X,Y \rangle := \tr( X^TY)$, for
$X,Y \in \mathbb{R}^{n\times k}$, also we identify
$T_{P}^{*}St(n,k)$ with $T_PSt(n,k)$, by the natural
isomorphism induced by this Riemannian metric. For more
details on the properties of the Stiefel manifold, see
\cite{absil2007optimization}. For each matrix
$A = [a_{ij}]$, denote
$\|A\|_{\beta} := (\sum_{i,j}|a_{ij}|^{\beta})^{1/\beta}$
and $A^-:=[\max \{-a_{ij},0\}].$

\begin{example} Let $G$ be a graph with vertices
  $V(G) := \{v_1,\ldots v_n\}$ and Edges $E(G)$. Fix integer
  $k>0$. Denote
  $D_k(G) := \{ (A_1, \ldots, A_k) \colon \emptyset \neq
  A_i\subset V(G), A_i \cap A_j = \emptyset, \text{ for all
    $i \neq j$}\}$ as the set of all $k$ sub-partitions of
  $V(G)$. Define $\partial A$, for a subset $A$ of $V(G)$,
  as the set of all edges of $G$ whose only one endpoint
  belongs to $A$. Now, consider the following Cheeger type
  constant of $G$:
\begin{equation}\label{eq:CheegerConstant}
  \gamma_{k}(G) := \min \sum_{i=1}^k \frac{|\partial A_i|}{\sqrt{|A_i|}},
\end{equation}
where the minimum runs over all
$(A_1\ldots,A_k) \in D_k(G)$. The constant $\gamma_k(G)$
indicates how well $G$ can be partitioned into $k$ clusters.
For more information on the Cheeger constant of graphs and
its applications to clustering, see
\cite{MR2644242,MR3831150,MR2409803}. For each
$\alpha := (\alpha_1,\ldots,\alpha_n) \in \mathbb{R}^{n}$,
define
$\|\nabla \alpha \|_1 := \sum_{\{v_i,v_j\} \in E(G)}|
\alpha_i-\alpha_j|$. By a technique similar to that of
Rothaus \cite{MR812396}, one can show\footnote{The key
  point is to show that for each vector
  $u := (u_1, \ldots, u_n)\in \mathbb{R}^n$ with
  $\|u\|_{2}=1$, there is a non-empty subset
  $A \subset \{ i | u_i \neq 0 \}$ such that
  $\frac{|\partial A|}{\sqrt{|A|}} \leq \|\nabla u \|_1$.}
that the discrete minimization problem
\eqref{eq:CheegerConstant}, is equivalent to the following
continuous optimization problem on a manifold with
non-negativity constraints:
\begin{equation}
  \label{eq:11}
  \min \sum_{i=1}^k \|\nabla u_i\|_1,
\end{equation}
where the minimum runs over all
$U := (u_1, \ldots, u_k) \in St(n,k)$ with non-negative
entries (which hereafter is denoted by $St_{+}(n,k)$). The
non-negativity constrains provide a combinatorial nature to the
feasible set of the problem \eqref{eq:11}. But, since the
cost function of \eqref{eq:11} is Lipschitzian on $St(n,k)$
with a rate $C$, the penalization lemma (see
\cite[Proposition 1.121]{PenotBook2013}) implies that the
problem \eqref{eq:11} is equivalent to the following
unconstrained optimization problem on a Stiefel manifold:
\begin{equation}
  \label{eq:12}
  \min_{U \in St(n,k)} \sum_{i=1}^k\|\nabla u_i\|_1 + C \dist(U; St_+(n,k)).
\end{equation}
Although computing the distance term in \eqref{eq:12} is not
a simple task, a simple observation shows that one can
replace the distance term with any upper estimate $h(U)$
which is zero if and only if the distance term is zero, or
equivalently, with a function $h(U)$ such that $0$ is the
minimum value of $h$, the set of minimizers of $h$ is
$St_+(n,k)$, and for every $U \in St(n,k)$,
\begin{equation}
  \label{eq:13}
  \dist(U; St_+(n,k)) \leq h(U).
\end{equation}
In other words, $St_+(n,k)$ should be the set of weak sharp
minima of the problem $\min_{St(n,k)}h(U)$ with $0$ as the
optimal value. Since the distance of a matrix
$U \in St(n,k)$ from the non-negative matrices in
the Euclidean space $\mathbb{R}^{n\times k}$ is a function of
the negative part of $U$, it seems that a natural candidate
for $h(U)$ is a function of $U^{-}$. So, we investigate the
following question:

Is $St_+(n,k)$ the set of weak sharp minima for the problem
of minimizing
$f_{\beta}(U):=\|U^-\|_{\beta}^{\beta} $
over the Stiefel manifold $St(n,k)$, for some $\beta > 0$?

According to Corollary
\ref{pro:NC-UnConWSP}, a necessary condition for provision of an
affirmative response to this question is that there
exists $\alpha>0$ such that, for every
$P \in St_+(n,k)$,\newcommand{\freS}{
  \hat{\partial}_{St(n,k)}} \newcommand{\norS}{
  \hat{N}_{St_+(n,k)}^{St(n,k)}} \newcommand{\norSR}{
  \hat{N}_{St_+(n,k)}^{\mathbb{R}^{n\times k}}}
\begin{equation}\label{eq:14}
  \alpha \mathbb{B}_{T_P St(n,k)} \cap \norS(P) \subset
  \freS f_{\beta}(P).
\end{equation}
Fix $P := [p_{ij}] \in St_{+}(n,k)$ and, for simplicity,
suppose that only the first $t$ rows of $P$ are non-zero
(for some $t \geq 0$). For each $n\times k$ matrix $A$,
Denote \newcommand{\topM}[1]{\widetilde{#1}} $\topM{A}$ by
the matrix whose rows are the first $t$ rows of $A$ and
\newcommand{\tailM}[1]{\widehat{#1}} $\tailM{A}$ by the
matrix whose rows are the last $n-t$ rows of $A$. According
to \eqref{eq:17}, we have $X \in \norS (P)$, if and only if
$X \in T_PSt(n,k)$ and
$
 \limsup _{}\langle X, U-P
\rangle / \| U - P \| \leq 0
$,
where the $\limsup$ runs over all $U$ approaching $P$, in
$St_{+}(n,k)$. By approaching $P$, from some suitable curves
in $St_+(n,k)$, one can see
\begin{equation}
  \label{eq:15}
  \norS (P) =
  \big\{
  X \in T_{P}St(n,k)\ \big| \ 
   \tailM{X} \leq 0, \topM{X} \circ \topM{P} = O_{t\times k}
  \big\},
\end{equation}
where $\circ$ is the entry-wise product, $O_{t\times k}$ is the
$t \times k$ zero matrix, and $\tailM{X} \leq 0$ means that
$\tailM{X}$ is a non-positive matrix. Next, we will
investigate when the necessary condition \eqref{eq:14}
holds. First, note that if $\beta > 1$, then $f_{\beta}$ is
smooth. So, the subdifferential $ \freS f_{\beta}(P)$ has
only one element. But, there are some $P \in St_{+}(n,k)$
such that the left hand of \eqref{eq:14} have infinitely
many elements. So, for $\beta > 1$, the necessary condition
\eqref{eq:14} does not hold.

Next, we will show that the necessary condition
\eqref{eq:14} holds, for each $ \beta <1$, with
$\alpha = 1$. Consider an arbitrary element
$X \in \mathbb{B}_{T_P St(n,k)} \cap \norS(P)$. We will show
$X \in \freS f_{\beta}(P)$. Define the smooth function
$\bar{g}(U)=\langle X, \exp_P^{-1}U \rangle$ on a sufficiently
small neighborhood of $P$. For simplicity, denote
$dU = \exp_{P}^{-1}U$. Note that since
$X \in \mathbb{B}_{T_P St(n,k)}$, we have
$\tr(X^TX) \leq 1$. So, the absolute value of the entries of
$\tailM{X}$ are at most equal to $1$. Moreover, $\tailM{X} \leq 0$. Thus,
$\langle \tailM{X}, \tailM{dU} \rangle \leq
\|\tailM{dU}^{-}\|_{1}$. Therefore, by the Cauchy--Schwarz
inequality, we have
\newcommand{\z}[1]{I_{#1}}
\begin{equation}\label{eq:9}
  \bar{g}(U) = \langle \topM{X} , \topM{dU} \circ \topM{\z{P}} \rangle + \langle \tailM{X} , \tailM{dU}\rangle  \leq \| \topM{dU} \circ \topM{\z{P}} \|_2 + \|\tailM{dU}^{-} \|_{1},
\end{equation}
where $\z{P}$ is the $n \times k$ matrix whose $(i,j)$th
entry is $1$ if $p_{ij}=0$ and is $0$ otherwise. One can see
that the function $\|\cdot^{-}\|_{1}$ is a norm on
$S :=\{ V \in \mathbb{R}^{t \times k} | \topM{P}^{T} V +
V^{T} \topM{P} = 0 \}$. Indeed, the function
$\|\cdot^{-}\|_1$ satisfies the triangle inequality, since
for each real $a,b$ we have $(a + b)^- \leq a^- + b^-$. So,
it is enough to show that $V = 0$ whenever $V \in S$ and
$\|V^{-}\|_1 = 0$. This follows from the non-negativity of
the entries of $P$ and that $P$ has no zero row. Now, since
$dU \in T_PSt(n,k)$, we have $\topM{dU} \in S$ and
$\topM{dU} \circ \topM{\z{P}} \in S$. By the equivalence of
norms on a finite dimensional linear space, we have
$\| \topM{dU} \circ \topM{\z{P}} \|_2 \leq C_P\|(\topM{dU}
\circ \topM{\z{P}})^-\|_1$, for some $C_{P} > 1$. Note that
$ \tailM{dU} = \tailM{dU} \circ \tailM{\z{P}}$. So, from
\eqref{eq:9}, we have
\begin{equation}\label{eq:16}
  \bar{g}(U) \leq
  C_{p}(\| (\topM{dU} \circ \topM{\z{P}})^{-} \|_1 +
  \|(\tailM{dU} \circ \tailM{\z{P}})^- \|_{1}) =
  C_P \| (dU \circ \z{P})^-\| _1.
\end{equation}
Now, consider the function
$g(U) := \bar{g}(U) +C_P\|(U\circ \z{P})^-\|_1- C_p\|(dU \circ
I_P)^-\|_1$. Since
$\|(U\circ \z{P})^-\|_1- \|(dU \circ I_P)^-\|_1$ is of order
$o(\|U-P\|)$ (note that $dU = (U-P) + o(\|U-P\|)$,
  the function $\|\cdot^{-}\|_1$ satisfies the triangle
  inequality, and $P \circ I_{P} = O_{n\times k}$),
 we have $dg(P) = d\bar{g}(P) = X$.
Therefore, as \eqref{eq:16}, we have 
\begin{eqnarray}
  g(U) &\leq& C_{P}\|(U \circ \z{P})^-\|_1
       \leq
               C_p \|U^-\|_1 = \|U^-\|_1^{\beta} (C_P\|U^-\|_1^{1-\beta}) \notag
  \\ &\leq& \|U^-\|_1^{\beta} \leq \|U^-\|_{\beta}^{\beta},\label{eq:18}
\end{eqnarray}
on a sufficiently small neighborhood of $P$, in which
$C_P\|U^-\|_1^{1-\beta} \leq 1$ (the last inequality of
\eqref{eq:18} follows from the fact that for each
$a,b \geq 0$ and $0< \beta <1$, we have $(a + b)^{\beta}
\leq a^{\beta} + b^{\beta}$). Thus, $g \leq f_{\beta}$
on a neighborhood of $P$.  Since
$g(P)=f_{\beta}(P)$ and $dg(P) = X$, as the
homotone property of the Fr\'echet subdifferential, we have
$X \in \freS f_{\beta}(P)$. \qed
\end{example}
\section{Concluding Remarks}
We presented a lemma and used it to derive some local
properties of a distance function on a Riemannian manifold
using the corresponding properties in linear space. In
this regard, we established a relation between the Fr\'echet
subdifferential (directional derivative) of a distance
function and a normal cone (contingent cone) on a Riemannian
manifold. Then, we established some primal and dual
necessary conditions for the set of weak sharp minima of
nonconvex optimization problems on Riemannian manifold.
 As an application, we showed how the notion of
weak sharp minima and our stated results can be used to model
a Cheeger type constant of a graph as an unconstrained
optimization problem on a Stiefel manifold. 

% Appendices go here. If there is one appendix the title is
% just Appendix. If there are appendices, the titles are
% Appendix A, Appendix B, and so on.

\begin{acknowledgements}
  The authors thank Sharif University of Technology for
  supporting this work. The first author is grateful to Amir
  Daneshgar, Mostafa Einollahzadeh, Mehdi
  Shaeiri, and  Mostafa Kiyaee for constructive discussions.
\end{acknowledgements}

% References BibTeX users please use
\bibliographystyle{spmpsci_unsrt} % mathematics and physical sciences
\bibliography{FrechetBib1} % name you BibTeX data base

\begin{thebibliography}{10}
\providecommand{\url}[1]{{#1}}
\providecommand{\urlprefix}{URL }
\expandafter\ifx\csname urlstyle\endcsname\relax
  \providecommand{\doi}[1]{DOI~\discretionary{}{}{}#1}\else
  \providecommand{\doi}{DOI~\discretionary{}{}{}\begingroup
  \urlstyle{rm}\Url}\fi

\bibitem{udriste1994convex}
Udri{\c s}te, C.: Convex {F}unctions and {O}ptimization {M}ethods on
  {R}iemannian {M}anifolds, \emph{Mathematics and its Applications}, vol. 297.
\newblock Kluwer Academic Publishers Group, Dordrecht (1994)

\bibitem{AbsilTrustRei2007}
Absil, P.A., Baker, C.G., Gallivan, K.A.: Trust-region methods on {R}iemannian
  manifolds.
\newblock Found. Comput. Math. \textbf{7}(3), 303--330 (2007)

\bibitem{absil2007optimization}
Absil, P.A., Mahony, R., Sepulchre, R.: Optimization {A}lgorithms on {M}atrix
  {M}anifolds.
\newblock Princeton University Press (2007)

\bibitem{Smith1994}
Smith, S.T.: Optimization techniques on {R}iemannian manifolds.
\newblock In: Hamiltonian and gradient flows, algorithms and control,
  \emph{Fields Inst. Commun.}, vol.~3, pp. 113--136. Amer. Math. Soc.,
  Providence, RI (1994)

\bibitem{HosseiniTrustRegion2016}
Grohs, P., Hosseini, S.: Nonsmooth trust region algorithms for locally
  {L}ipschitz functions on {R}iemannian manifolds.
\newblock IMA Journal of Numerical Analysis \textbf{36}(3), 1167--1192 (2016)

\bibitem{AdlerNewtonSpin2002}
Adler, R.L., Dedieu, J.P., Margulies, J.Y., Martens, M., Shub, M.: Newton's
  method on {R}iemannian manifolds and a geometric model for the human spine.
\newblock IMA J. Numer. Anal. \textbf{22}(3), 359--390 (2002)

\bibitem{EdelmanStiefel1999}
Edelman, A., Arias, T.A., Smith, S.T.: The geometry of algorithms with
  orthogonality constraints.
\newblock SIAM J. Matrix Anal. Appl. \textbf{20}(2), 303--353 (1999)

\bibitem{MR3160322}
Dong, X., Frossard, P., Vandergheynst, P., Nefedov, N.: Clustering on
  multi-layer graphs via subspace analysis on {G}rassmann manifolds.
\newblock IEEE Trans. Signal Process. \textbf{62}(4), 905--918 (2014)

\bibitem{MR2678395}
Journ\'ee, M., Bach, F., Absil, P.A., Sepulchre, R.: Low-rank optimization on
  the cone of positive semidefinite matrices.
\newblock SIAM J. Optim. \textbf{20}(5), 2327--2351 (2010)

\bibitem{MR2913705}
Shalit, U., Weinshall, D., Chechik, G.: Online learning in the embedded
  manifold of low-rank matrices.
\newblock J. Mach. Learn. Res. \textbf{13}, 429--458 (2012)

\bibitem{MR3604647}
Sun, J., Qu, Q., Wright, J.: Complete dictionary recovery over the sphere {II}:
  recovery by {R}iemannian trust-region method.
\newblock IEEE Trans. Inform. Theory \textbf{63}(2), 885--914 (2017)

\bibitem{MR3069099}
Vandereycken, B.: Low-rank matrix completion by {R}iemannian optimization.
\newblock SIAM J. Optim. \textbf{23}(2), 1214--1236 (2013)

\bibitem{MR2811294}
Ishteva, M., Absil, P.A., Van~Huffel, S., De~Lathauwer, L.: Best low
  multilinear rank approximation of higher-order tensors, based on the
  {R}iemannian trust-region scheme.
\newblock SIAM J. Matrix Anal. Appl. \textbf{32}(1), 115--135 (2011)

\bibitem{MR3123829}
Mishra, B., Meyer, G., Bach, F., Sepulchre, R.: Low-rank optimization with
  trace norm penalty.
\newblock SIAM J. Optim. \textbf{23}(4), 2124--2149 (2013)

\bibitem{Bento2015nonconPoriximal}
Bento, G.C., Ferreira, O.P., Oliveira, P.R.: Proximal point method for a
  special class of nonconvex functions on {H}adamard manifolds.
\newblock Optimization \textbf{64}(2), 289--319 (2015)

\bibitem{MR3606427}
Hosseini, S., Uschmajew, A.: A {R}iemannian gradient sampling algorithm for
  nonsmooth optimization on manifolds.
\newblock SIAM J. Optim. \textbf{27}(1), 173--189 (2017)

\bibitem{BURKE1992}
Burke, J.V., Ferris, M.C., Qian, M.: On the {C}larke subdifferential of the
  distance function of a closed set.
\newblock J. Math. Anal. Appl. \textbf{166}(1), 199--213 (1992)

\bibitem{MR2179254}
Mordukhovich, B.S., Nam, N.M.: Subgradient of distance functions with
  applications to {L}ipschitzian stability.
\newblock Math. Program. \textbf{104}(2-3, Ser. B), 635--668 (2005)

\bibitem{LiMordukhovichWangYao}
Li, C., Mordukhovich, B.S., Wang, J., Yao, J.C.: Weak sharp minima on
  {R}iemannian manifolds.
\newblock SIAM J. Optim. \textbf{21}(4), 1523--1560 (2011)

\bibitem{PolyakSharpMinima1979}
Polyak, B.T.: {S}harp {M}inima, {I}nstitute of {C}ontrol {S}ciences {L}ecture
  {N}otes, {M}oscow, {USSR}, 1979; presented at the {IIASA} {W}orkshop on
  {G}eneralized {L}agrangians and {T}heir {A}pplications, {IIASA}, {L}axenburg,
  {A}ustria, 1979.

\bibitem{FerrisThesis1988}
Ferris, M.C.: Weak sharp minima and penalty functions in mathematical
  programming, {Ph.D.} thesis, {Universiy of Cambridge}, {Cambridge}, 1988

\bibitem{BurkeWeakSharpII2005}
Burke, J.V., Deng, S.: Weak sharp minima revisited. {II}. {A}pplication to
  linear regularity and error bounds.
\newblock Math. Program. \textbf{104}(2-3, Ser. B), 235--261 (2005)

\bibitem{BurkeFerris1993}
Burke, J.V., Ferris, M.C.: Weak sharp minima in mathematical programming.
\newblock SIAM J. Control Optim. \textbf{31}(5), 1340--1359 (1993)

\bibitem{Mordu2006FreSubCal}
Mordukhovich, B.S., Nam, N.M., Yen, N.D.: Fr{\'e}chet subdifferential calculus
  and optimality conditions in nondifferentiable programming.
\newblock Optimization \textbf{55}(5-6), 685--708 (2006)

\bibitem{NgWeaksharp2003}
Ng, K.F., Zheng, X.Y.: Global weak sharp minima on {B}anach spaces.
\newblock SIAM J. Control Optim. \textbf{41}(6), 1868--1885 (2003)

\bibitem{Ward1999}
Studniarski, M., Ward, D.E.: Weak sharp minima: characterizations and
  sufficient conditions.
\newblock SIAM J. Control Optim. \textbf{38}(1), 219--236 (1999)

\bibitem{Ward1994}
Ward, D.E.: Characterizations of strict local minima and necessary conditions
  for weak sharp minima.
\newblock Journal of Optimization Theory and Applications \textbf{80}(3),
  551--571 (1994)

\bibitem{MordukhovichWeaksharp2012}
Zhou, J., Mordukhovich, B.S., Xiu, N.: Complete characterizations of local weak
  sharp minima with applications to semi-infinite optimization and
  complementarity.
\newblock Nonlinear Anal. \textbf{75}(3), 1700--1718 (2012)

\bibitem{MR2988725}
Hosseini, S., Pouryayevali, M.R.: On the metric projection onto prox-regular
  subsets of {R}iemannian manifolds.
\newblock Proc. Amer. Math. Soc. \textbf{141}(1), 233--244 (2013)

\bibitem{ZhuLedyaev2007}
Ledyaev, Y.S., Zhu, Q.J.: Nonsmooth analysis on smooth manifolds.
\newblock Trans. Amer. Math. Soc. \textbf{359}(8), 3687--3732 (2007)

\bibitem{PenotBook2013}
Penot, J.P.: Calculus {W}ithout {D}erivatives, \emph{Graduate Texts in
  Mathematics}, vol. 266.
\newblock Springer, New York (2013)

\bibitem{AzagraFerreraLopez2005}
Azagra, D., Ferrera, J., L{\'o}pez-Mesas, F.: Nonsmooth analysis and
  {H}amilton-{J}acobi equations on {R}iemannian manifolds.
\newblock J. Funct. Anal. \textbf{220}(2), 304--361 (2005)

\bibitem{Pavel1982}
Motreanu, D., Pavel, N.: Quasi-tangent vectors in flow-invariance and
  optimization problems on {B}anach manifolds.
\newblock Journal of Mathematical Analysis and Applications \textbf{88}(1), 116
  -- 132 (1982)

\bibitem{MR2569498}
Grigor'yan, A.: Heat kernel and analysis on manifolds, \emph{AMS/IP Studies in
  Advanced Mathematics}, vol.~47.
\newblock American Mathematical Society, Providence, RI; International Press,
  Boston, MA (2009)

\bibitem{MR1335233}
Lang, S.: Differential and {R}iemannian {M}anifolds, \emph{Graduate Texts in
  Mathematics}, vol. 160, third edn.
\newblock Springer-Verlag, New York (1995)

\bibitem{MordukhovichBookI2006}
Mordukhovich, B.S.: Variational {A}nalysis and {G}eneralized {D}ifferentiation.
  {I}, \emph{Grundlehren der Mathematischen Wissenschaften [Fundamental
  Principles of Mathematical Sciences]}, vol. 330.
\newblock Springer-Verlag, Berlin (2006)
\bibitem{MR1942573}
N\'{e}meth, S.Z.: Variational inequalities on {H}adamard manifolds.
\newblock Nonlinear Anal. \textbf{52}(5), 1491--1498 (2003)

\bibitem{MR3192427}
Krist\'{a}ly, A.: Nash-type equilibria on {R}iemannian manifolds: a variational
  approach.
\newblock J. Math. Pures Appl. (9) \textbf{101}(5), 660--688 (2014)

\bibitem{MR2644242}
Daneshgar, A., Hajiabolhassan, H., Javadi, R.: On the isoperimetric spectrum of
  graphs and its approximations.
\newblock J. Combin. Theory Ser. B \textbf{100}(4), 390--412 (2010)

\bibitem{MR3831150}
Tudisco, F., Hein, M.: A nodal domain theorem and a higher-order {C}heeger
  inequality for the graph {$p$}-{L}aplacian.
\newblock J. Spectr. Theory \textbf{8}(3), 883--908 (2018)

\bibitem{MR2409803}
von Luxburg, U.: A tutorial on spectral clustering.
\newblock Stat. Comput. \textbf{17}(4), 395--416 (2007)

\bibitem{MR812396}
Rothaus, O.S.: Analytic inequalities, isoperimetric inequalities and
  logarithmic {S}obolev inequalities.
\newblock J. Funct. Anal. \textbf{64}(2), 296--313 (1985)
\end{thebibliography}

\end{document}